\documentclass{birkjour}
 \newtheorem{thm}{Theorem}[section]
 \newtheorem{cor}[thm]{Corollary}
 \newtheorem{lem}[thm]{Lemma}
 
 \theoremstyle{definition}
 \newtheorem{defn}[thm]{Definition}
 \theoremstyle{remark}
 \newtheorem{rem}[thm]{Remark}
 \newtheorem*{ex}{Example}
 \numberwithin{equation}{section}

\begin{document}

%-------------------------------------------------------------------------
% editorial commands: to be inserted by the editorial office
%
%\firstpage{1} \volume{228} \Copyrightyear{2004} \DOI{003-0001}
%\seriesextraline{This is the Concrete Title of this Book\br H.E. R and S.T.C. W, Eds.}
%\firstpage{1}

\title[On a type of Static Equation on Certain Contact Metric Manifolds]
 {On a type of Static Equation on Certain Contact Metric Manifolds}

%----------Author 1
\author[Mohan Khatri]{Mohan Khatri}

\address{%
Department of Applied Mathematics and Statistics\\
Pachhunga University College\\
Aizawl-796001\\
India}

\email{mohankhatri.official@gmail.com}

%\thanks{This work was completed with the support of our
%----------Author 2
\author{Jay Prakash Singh}
\address{%
Department of Mathematics\\
Central University of South Bihar\\
Gaya-824236, Bihar\\ India}
\email{jpsmaths@gmail.com}


\date{January 1, 2004}

\begin{abstract}
This paper deals with the investigation of $K$-contact and $(\kappa,\mu)$-contact manifolds admitting a positive smooth function $f$ satisfying the equation: $$f\mathring{Ric}=\mathring{\nabla}^2f$$ where $\mathring{Ric}$, $\mathring{\nabla}^2f$ are traceless Ricci tensor and Hessian tensor respectively. We proved that if a complete and simply connected $K$-contact manifold admits such a smooth function $f$, then it is isometric to the unit sphere $\mathbb{S}^{2n+1}$. Next, we showed that if a non-Sasakian $(\kappa,\mu)$-contact metric manifold admit such a smooth function $f$, then it is locally flat for $n=1$ and for $n>1$ is locally isometric to the product space $E^{n+1}\times S^n(4)$.
\end{abstract}

%%% ----------------------------------------------------------------------
\maketitle
%%% ----------------------------------------------------------------------
%\tableofcontents
\section{Introduction}
Static space-times are special and important global solutions to Einstein equations which show the interplay between matter and space-time in general relativity. A static space-time metric $\hat{g}=-f^2dt^2+g$ satisfying the Einstein equation
\begin{eqnarray}\label{1.1}
Ric_{\hat{g}}-\frac{R_{\hat{g}}}{2}=T,
\end{eqnarray}
for the energy-momentum stress tensor $T=-\mu f^2dt^2-\rho g$ of a perfect fluid, where the smooth functions $\mu$ and $\rho$ are the energy density and pressure of the perfect fluid respectively. Moreover, $Ric_{\hat{g}}$ and $R_{\hat{g}}$ stand for Ricci tensor and the scalar curvature for the metric $\hat{g}$. For details see \cite{1,2,3,4}. The Friedmann-Lemaitre-Robertson-Walker solutions for the Einstein equation with perfect fluid as a matter fluid represents a homogeneous, fluid filled universe that is undergoing accelerated expansion \cite{5}. Therefore static perfect fluid space-times acts as a generalizations of the static vacuum spaces and hence is an important topic of study both for physicians and mathematicians.\\

We now recall the definition of static perfect fluid space-time (see \cite{6,3,7}).
\begin{defn}
A Riemannian manifold $(M^n,g)$ is said to be the spatial factor of a static perfect fluid space-time if there exist smooth functions $f>0$ and $\rho$ on $M^n$ satisfying the perfect fluid equations:
\begin{eqnarray}\label{1.2}
f\mathring{Ric}=\mathring{\nabla}^2f
\end{eqnarray}
and 
\begin{eqnarray}\label{1.3}
\Delta f=\Big(\frac{n-2}{2(n-1)}R+\frac{n}{n-1}\rho \Big)f,
\end{eqnarray}
where $\mathring{Ric}$ and $\mathring{\nabla}^2$ stand for the traceless Ricci and Hessian tensors, respectively. When $M^n$ has boundary $\partial M$, we assume in addition that $f^{-1}(0)=\partial M$.
\end{defn}
Coutinho et al. \cite{6} obtained a sharp boundary estimate for static perfect fluid space-time on an $n$-dimensional compact manifold with boundary. Kobayashi and Obata \cite{3} gave a classification of conformally flat Riemannian manifold in which $(g,f)$ satisfies (\ref{1.2}). Recently, Leandro and Sol$\acute{o}$rzano \cite{8} investigate the static perfect fluid space-time $M^4\times_f\mathbb{R}$ such that $(M^4,g)$ is a half conformally flat Riemannian manifold and proved that $(M^4,g)$ is locally isometric to a warped product manifold $I\times_\phi N^3$ where $I\subset\mathbb{R}$ and $N^3$ is a space form. As illustrated by Coutinho et al. \cite{6} Schwarzchild space, which serves as a model for a static black hole, is an exact solution of the above equations. 

Let $(M^n,g)$ be a compact, oriented, connected Riemannian manifold with dimension $n$ at least three, $\mathcal{M}$ be the set of Riemannian metrics on $M^n$ of unitary volume, $\mathcal{C}\subset\mathcal{M}$ be the set of Riemannian metrics with constant scalar curvature. Define the total scalar curvature functional $\mathcal{R}:\mathcal{M}\rightarrow\mathbb{R}$ as $$\mathcal{R}(g)=\int_{M^n}R_gdM_g,$$ where $R_g$ is the scalar curvature on $M^n$. It is well-known that the formal $L^2$-adjoint of the linearization of the scalar curvature operator $\mathfrak{L}_g$ at $g$ is defined as $$\mathfrak{L}_g^*(f):=-(\Delta_gf)g+Hess_gf-fRic_g,$$ where $f$ is a smooth function on $M^n$, and $\Delta_g$, $Hess_g$ and $Ric_g$ stand for the Laplacian, the Hessian form and the Ricci curvature tensor on $M^n$, respectively. As in \cite{9,10} we say that $g$ is a $V$-static metric if there exists a smooth function $f$ on $M^n$ and a constant $\kappa$ satisfying 
\begin{eqnarray}\label{1.4}
\mathfrak{L}_g^*(f)=\kappa g.
\end{eqnarray}
Coutinho et al. \cite{6} consider Riemannian manifolds satisfying (\ref{1.2}) and obtained several classification on the compact case. Moreover, if the scalar curvature is constant, then $(M^n,g,f)$ satisfies $V$-static equation (Proposition 6, \cite{6}). Also, the Euler-Lagrangian equation of Hilbert-Einstein action on the space of Riemannian metric with unit volume and constant scalar curvature is
\begin{eqnarray}\label{1.5}
Ric-\frac{1}{n}Rg=\nabla^2f-\Big(Ric-\frac{1}{n-1}Rg\Big)f.
\end{eqnarray}
A Riemannian manifold $(M^n,g)(n\geq3)$ of constant scalar curvature is called CPE if it admits a smooth solution $f$ satisfying (\ref{1.5}). It is interesting to note that equations (\ref{1.2}) and (\ref{1.5}) are somewhat related. When $\kappa=0$ in equation (\ref{1.4}) together with smooth boundary $\partial M$ such that $f^{-1}(0)=\partial M$ is known as vacuum static space (Fishcher-Marsden equation \cite{11}) and when $\kappa=1$ in (\ref{1.4}) is known as Miao-Tam critical metric \cite{12}. The significance of equation (\ref{1.2}) is that this rearrangement enclosed a large group of metrics, such as static metric, CPE metric and Miao-Tam critical metric.\\

Several authors started investigating the above metrics on different contact structures. In \cite{13}, Ghosh and Patra investigated CPE in the framework of $K$-contact manifold and $(\kappa,\mu)$-contact manifold. They showed that the complete $K$-contact metric satisfying CPE is Einstein and is isometric to unit sphere $S^{2n+1}$. The same authors also considered Miao-Tam critical metric on contact geometry \cite{14}. In \cite{15}, they investigated CPE in Kenmotsu and Almost Kenmotsu manifolds. Recently, Kumara et al. \cite{16} investigated the characteristics of static perfect fluid space-time metrics on almost Kenmotsu manifolds. \\

In this paper, we would be investigating certain contact manifolds $(M^n,g), n\geq 3$ admitting a smooth non-trivial function $f$ satisfying
\begin{eqnarray}\label{A1}
f\mathring{Ric}=\mathring{\nabla}^2f.
\end{eqnarray}
It is clear that (\ref{A1}) is similar to the first part of (\ref{1.2}) without perfect fluid matter. A compact manifold admitting smooth function $f$ satisfying (\ref{A1}) was considered by Coutinho et al. \cite{6} The advantage of this arrangement, is that it encloses a large class of metric, for instance, the static perfect fluid space-time metric eq. (\ref{1.2}) and (\ref{1.3}) \cite{6,3,7}, the critical point equation eq. (\ref{1.5}) \cite{C1}, critical metrics of the volume functional \cite{9,11,12, De} and static spaces \cite{Am}.  
We proved that if a complete and simply connected $K$-contact manifold admitting smooth function $f$ satisfying (\ref{A1}) then it is isometric to the unit sphere $\mathbb{S}^{2n+1}$ and verified this by constructing an example. Moreover, a non-Sasakian $(\kappa,\mu)$-contact metric manifold admitting such function $f$ satisfying (\ref{A1}) is considered.
\section{Preliminaries}
A 2n+1-dimensional smooth manifold $M$ is said to have a contact structure if it admits a (1,1)-tensor field $\varphi$, a vector field $\xi$ called the characteristic vector field such that $d\eta(\xi,X)=0$ for every vector field $X$ on $M$, a 1-form $\eta$ such that $\eta\wedge(d\eta)^n\neq0$ everywhere and an  associate metric $g$ called Riemannian metric satisfying the following conditions:
\begin{eqnarray}\label{a1}
\varphi^2=-I+\eta\otimes\xi,~~~~~d\eta(X,Y)=g(X,\varphi Y),~~~~~\eta(X)=g(X,\xi),
\end{eqnarray}
\begin{eqnarray}\label{a2}
g(\varphi X,\varphi Y)=g(X,Y)-\eta(X)\eta(Y),
\end{eqnarray}
for any vectors field $X,Y$ on $M$. Moreover, if $\nabla$ denotes the Riemannian connection of $g$, then the following relation holds:
\begin{eqnarray}\label{a3}
\nabla_X\xi=-\varphi X-\varphi hX.
\end{eqnarray}
From the definition, it persues that $\varphi\xi=0$ and $\eta \circ\varphi=0$. Then, the manifold $M(\varphi,\xi,\eta,g)$ equipped with such a structure is called a contact metric manifold \cite{4a,7a}.\\

Given a contact metric manifold $M$ we define a symmetric (1,1)-tensor field $h$ and self adjoint operator $l$ by $h=\frac{1}{2}\mathcal{L}_\xi\varphi$ and $l=R(.,\xi)\xi$, where $\mathcal{L}$ denotes Lie differentiation. Then, $h\varphi=-\varphi h,~~~Trh=Tr\varphi h=0,~~~~h\xi=0.$ Also from \cite{7a},
\begin{eqnarray}\label{q1}
Ric(\xi,\xi)=g(Q\xi,\xi)=Trl=2n-|h|^2.
\end{eqnarray}
\begin{eqnarray}\label{m1}
\nabla_\xi h=\varphi-\varphi h^2-\varphi l.
\end{eqnarray}
A contact metric structure on $M$ is said to be normal if the almost complex structure on $M\times\mathbb{R}$ defined by $J(X,fd/dt)=(\varphi X-f\xi,\eta(X)d/dt)$, where $f$ is a real function on $M\times\mathbb{R}$, is integrable.
 A normal contact metric manifold is a Sasakian manifold. An almost contact metric manifold is Sasakian if and only if
\begin{eqnarray}\label{a4}
(\nabla_X\varphi)Y=g(X,Y)\xi-\eta(Y)X,
\end{eqnarray}
for any $X,Y \in TM$. The vector field $\xi$ is a killing vector with respect to $g$ if and only if $h=0.$ A contact metric manifold $M(\varphi,\xi,\eta,g)$ for which $\xi$ is killing (equivalently $h=0$ or $Trl=2n$) is said to be a K-contact metric manifold. On a K-contact manifold, the following formulas are known \cite{7a}
\begin{eqnarray}\label{k1}
\nabla_X\xi=-\varphi X,
\end{eqnarray}
\begin{eqnarray}\label{k2}
Q\xi=2n\xi,
\end{eqnarray}
\begin{eqnarray}\label{k3}
R(\xi,X)Y=(\nabla_X\varphi)Y,
\end{eqnarray}
where $\nabla$ is the operator of covarient differentiation of $g$, $Ric$ is the Ricci tensor of type (0,2) such that $Ric(X,Y)=g(QX,Y)$, where $Q$ is Ricci operator and $R$ is the Riemann curvature tensor of $g$. A Sasakian manifold is $K$-contact and the converse is not true except in dimension 3. The following formula also holds for a $K$-contact manifold
\begin{eqnarray}\label{k4}
(\nabla_Y\varphi)X+(\nabla_{\varphi Y}\varphi)\varphi X=2g(Y,X)\xi-\eta(X)(Y+\eta(Y)\xi).
\end{eqnarray}
\\
As a generalization of the Sasakian case, Blair et al. \cite{5a} introduced $(\kappa,\mu)$-nullity distribution on a contact metric manifold and gave several reasons for studying it. A full classification of $(\kappa,\mu)$-spaces was given by Boeckx \cite{6a}.\\

The $(\kappa,\mu)$-nullity distribution of a contact metric manifold $M^{2n+1}(\varphi,\xi,\eta,g)$ is a distribution
\begin{eqnarray}
N(\kappa,\mu):p\rightarrow N_p(\kappa,\mu)=\{Z\in T_pM:R(X,Y)Z\nonumber=\kappa\{g(Y,Z)X\\-g(X,Z)Y\}+\mu\{g(Y,Z)hX-g(X,Z)hY\}\},\nonumber
\end{eqnarray}
for any $X,Y,Z \in T_pM$ and real numbers $\kappa$ and $\mu$. A contact metric manifold $M^{2n+1}$ with $\xi \in N(\kappa,\mu)$ is called a $(\kappa,\mu)$-contact metric manifold. In particular, if $\mu=0$, then the notion of $(\kappa,\mu)$-nullity distribution reduces to the notion of $\kappa$-nullity distribution, introduced by Tanno \cite{8a}. If $\kappa=1$, the structure is Sasakian, and if $\kappa<1$, the $(\kappa,\mu)$-nullity condition determines the curvature of the manifold completely.\\

\noindent In a $(\kappa,\mu)$-contact metric manifold the following relations hold \cite{5a,3a}
\begin{eqnarray}\label{a5}
h^2=(k-1)\varphi^2,~~~~~k\leq 1,
\end{eqnarray}
\begin{eqnarray}\label{a8}
R(X,Y)\xi=k[\eta(Y)X-\eta(X)Y]+\mu[\eta(Y)hX-\eta(X)hY],
\end{eqnarray}
\begin{eqnarray}\label{a9}
QX=[2(n-1)-n\mu]X+[2(n-1)+\mu]hX\nonumber\\+[2(1-n)+n(2k+\mu)]\eta(X)\xi,
\end{eqnarray}
\begin{eqnarray}\label{a11}
R=2n(2n-2+k-n\mu).
\end{eqnarray}
 Here, $R$ is the scalar curvature of the manifold. 
 \section{Main results}
 In this section we characterized $K$-contact manifold and $(\kappa,\mu)$-contact manifold in which $f$ a smooth function satisfies (\ref{A1}). First, we prove the following result which will be used later on.
 \begin{lem}\label{l1}
 If a Riemannian manifold $(M^{2n+1},g)$ admits a smooth non-trivial function $f$ satisfying (\ref{A1}), then the curvature tensor $R$ can be expressed as
 \begin{eqnarray}
 R(X,Y)Df&=(Xf)QY-(Yf)QX+f[(\nabla_XQ)Y\nonumber\\&-(\nabla_YQ)X]+(X\psi)Y-(Y\psi)X,\nonumber
 \end{eqnarray}
 for any vector fields $X,Y$ on $M^{2n+1}$ and a smooth function $\psi=\frac{\Delta f-Rf}{2n+1}$.
 \end{lem} 
 \begin{proof}
 Suppose a Riemannian manifold $(M^{2n+1},g)$ admits a smooth non-trivial function $f$ satisfying (\ref{A1}), then (\ref{A1}) can be rewritten as
 \begin{eqnarray}\label{2.1}
 \nabla_XDf=fQX+\psi X,
 \end{eqnarray}
 Taking the covariant derivative of (\ref{2.1}) along arbitrary vector field $Y$ we obtained
 \begin{eqnarray}\label{2.2}
 \nabla_Y\nabla_XDf=(Yf)QX+f(\nabla_YQ)X\nonumber\\+fQ(\nabla_YX)+(Y\psi)X+\psi(\nabla_YX).
 \end{eqnarray}
Inserting repeatedly (\ref{2.2}) in the expression for curvature tensor, $$R(X,Y)Z=\nabla_X\nabla_YZ-\nabla_Y\nabla_XZ-\nabla_{[X,Y]}Z$$ we get the required result.
 \end{proof}
 \begin{thm}\label{t1}
 Let $M^{2n+1}(\varphi,\eta,\xi,g)$ be a complete and simply connected $K$-contact manifold and $f$ a non-trivial smooth function satisfying (\ref{A1}), then it is compact, Einstein and isometric to the unit sphere $\mathbb{S}^{2n+1}$.
 \end{thm}
\begin{proof}
Taking the covariant derivative of (\ref{k2}) along arbitrary vector field $X$ and using (\ref{k1}) gives
\begin{eqnarray}\label{2.3}
(\nabla_XQ)\xi=Q\varphi X-2n\varphi X,
\end{eqnarray}
for any vector field $X$ on $M$. As $\xi$ is Killing, we have $\mathcal{L}_\xi Q=0$. Making use of (\ref{k2}) and (\ref{2.3}) in the last expression yields
\begin{eqnarray}\label{2.4}
\nabla_\xi Q=Q\varphi-\varphi Q.
\end{eqnarray}
Replacing $X$ by $\xi$ in Lemma \ref{l1} and using (\ref{2.3}), (\ref{2.4}) and (\ref{k2}) we get
\begin{eqnarray}\label{2.5}
R(\xi,X)Df=(\xi f)QX-2n(Xf)\xi\nonumber\\+f[2n\varphi X-\varphi QX]+(\xi\psi)X-(X\psi)\xi.
\end{eqnarray}
Taking an inner product of (\ref{2.5}) with vector field $Y$, then using (\ref{k3}) after taking an inner product of (\ref{k3}) with $Df$, we obtain:
\begin{eqnarray}\label{2.6}
g((\nabla_X\varphi)Y,Df)+(\xi f)Ric(X,Y)-2n(Xf)\eta(Y)+f\{2ng(\varphi X,Y)\nonumber\\-g(\varphi QX,Y)\}+(\xi\psi)g(X,Y)-(X\psi)\eta(Y)=0.
\end{eqnarray}
Replacing $X$ by $\varphi X$ and $Y$ by $\varphi Y$ in (\ref{2.6}) yields
\begin{eqnarray}\label{2.7}
g((\nabla_{\varphi X}\varphi)\varphi Y,Df)+(\xi f)Ric(\varphi X,\varphi Y)\nonumber\\-2nfg(X,\varphi Y)-fg(Q\varphi X,Y)+(\xi \psi)g(\varphi X,\varphi Y)=0.
\end{eqnarray}
Making use of (\ref{k4}) and (\ref{2.6}) in (\ref{2.7}) we get
\begin{eqnarray}\label{2.8}
(2n-1)\{(Xf)\eta(Y)-(Yf)\eta(X)\}+\{(Y\psi)\eta(X)\nonumber\\-(X\psi)\eta(Y)\}+8nfg(\varphi X,Y)+2fg((Q\varphi+\varphi Q)Y,X)=0.
\end{eqnarray}
Taking $X=\varphi X$ and $Y=\varphi Y$ in (\ref{2.8}) and using the fact that $f\neq0$, we get
\begin{eqnarray}\label{2.9}
(Q\varphi+\varphi Q)X=4n\varphi X,
\end{eqnarray}
for any vector field $X$ on $M$.\\ Let $\{e_i,\varphi e_i,\xi\}, i=1,2,...n$ be a $\varphi$-basis of $M$ such that $Qe_i=\lambda_ie_i$. Then we have $\varphi Qe_i=\lambda_i\varphi e_i$. Making use of this in (\ref{2.9}), we get $Q\varphi e_i=(4n-\lambda_i)\varphi e_i$. Therefore, the scalar curvature is given as
\begin{eqnarray}
R=g(Q\xi,\xi)+\sum_{i=1}^n[g(Qe_i,e_i)+g(Q\varphi e_i,\varphi e_i)]=2n(2n+1).\nonumber
\end{eqnarray}
Taking an inner product of (\ref{2.5}) with $Df$, then using (\ref{k2}) gives
\begin{eqnarray}\label{2.10}
(\xi f)QDf-2n(Df)(\xi f)+f\{Q\varphi Df\nonumber\\-2n\varphi Df\}+(\xi\psi)Df-(D\psi)(\xi f)=0.
\end{eqnarray}
Again taking an inner product of (\ref{2.5}) with $\xi$ then using $R(\xi,X)\xi=\eta(X)\xi-X$, we obtain:
\begin{eqnarray}\label{2.11}
(2n+1)[Df-(\xi f)\xi]=[(\xi\psi)-D\psi].
\end{eqnarray}
Taking the covariant derivative of (\ref{2.11}) along arbitrary vector field $X$, then taking an inner product of forgoing equation with vector field $Y$, we get
\begin{eqnarray}\label{2.12}
2(2n+1)(\xi f)g(\varphi X,Y)+2(\xi\psi)g(\varphi X,Y)\nonumber\\=X[(2n+1)(\xi f)+(\xi\psi)]\eta(Y)-Y[(2n+1)(\xi f)+(\xi\psi)]\eta(X),
\end{eqnarray}
where we make use of $g(\nabla_XDf,Y)=g(\nabla_YDf,X)$. Choosing $X,Y\perp \xi$ and noticing that $d\eta\neq 0$ in (\ref{2.12}) gives $(2n+1)(\xi f)+(\xi\psi)=0$, hence (\ref{2.11}) becomes
\begin{eqnarray}\label{2.13}
(2n+1)Df+D\psi=0.
\end{eqnarray}
Making use of (\ref{2.9}) and (\ref{2.13}) in (\ref{2.10}) gives
\begin{eqnarray}\label{2.14}
(\xi f)[QDf-2nDf]-f[\varphi QDf-2n\varphi Df]=0.
\end{eqnarray}
Operating (\ref{2.14}) by $\varphi$ yields
\begin{eqnarray}\label{2.15}
(\xi f)[\varphi QDf-2n\varphi Df]-f[\varphi^2 QDf-2n\varphi^2 Df]=0.
\end{eqnarray}
Combining (\ref{2.15}), (\ref{2.14}) and using (\ref{a1}) we obtain
\begin{eqnarray}\label{2.16}
((\xi f)^2+f^2)(QDf-2nDf)=0.
\end{eqnarray}
As $fneq0$ we must have $QDf=2nDf$. Taking the covariant derivative of last equation and using (\ref{2.1}) gives $$(\nabla_XQ)Df+fQ^2X+(\psi-2nf)QX-2n\psi X=0.$$ Contracting the obtained equation implies $|Q|^2=2nR$, as $R=2n(2n+1)$ is constant. Then as a consequence, we get, $$\Big|Q-\frac{R}{2n}I\Big|^2=|Q|^2-\frac{2R^2}{2n+1}+\frac{R^2}{2n+1}=0.$$
Therefore, we must have $Q=\frac{R}{2n+1}I=2nI$, that is, $M$ is Einstein. As $M$ is complete, by Myer's theorem \cite{1b} it is compact. Integrating (\ref{2.13}), we get $\psi=-(2n+1)f+k$, where $k$ is some constant. In consequence of this, and $Q=2nI$ in (\ref{2.1}) yields $$\nabla_XDf=(-f+k)X.$$ We now apply Tashiro's theorem \cite{2b} to conclude that $M$ is isometric to the unit sphere $\mathbb{S}^{2n+1}$. This completes the proof.
\end{proof}

On a Sasakian manifold, it is well-known that the Ricci operator commute with $\varphi$, that is, $Q\varphi=\varphi Q$.
Making use of this in (\ref{2.9}) gives $Q\varphi X=2n\varphi X$. Then replacing $X$ by $\varphi X$ in the last expression yield $QX=2nX$, which implies that $M$ is Einstein. Therefore proceeding similarly as in Theorem \ref{t1}, we can state the following:
\begin{cor}\label{c1}
Let $M^{2n+1}(\varphi,\eta,\xi,g)$ be a complete and simply connected Sasakian manifold and $f$ a non-trivial smooth function satisfying (\ref{A1}), then it is compact, Einstein and isometric to the unit sphere $\mathbb{S}^{2n+1}$.
\end{cor}

It is known that if a Riemannian manifold $(M^n,g)(n\geq 3)$ with constant scalar curvature and $f$ a smooth function on $M^n$ satisfying (\ref{1.2}). Then $(M^n,g,f)$ satisfies the $V$-static equation (\ref{1.4}) for some constant $k$ (Proposition 6, \cite{6}). Further, we see that the scalar curvature of a $K$-contact manifold is constant from Theorem \ref{t1} and hence we are in a position to state the following:
\begin{cor}\label{c2}
 Let $M^{2n+1}(\varphi,\eta,\xi,g)$ be a complete and simply connected $K$-contact manifold without boundary and $f$ a non-trivial smooth-function satisfying $V$-static equation (\ref{1.4}), then it is compact, Einstein and isometric to the unit sphere $S^{2n+1}$.
\end{cor}
\begin{rem}\label{r1}
In particular, the $V$-static equation reduces to Fishcher-Marsden equation \cite{11} for $k=0$ and  Miao-Tam critical metric \cite{12} for $k=1$. In \cite{14}, authors studied Miao-Tam critical metric and in \cite{3b} the Fishcher-Marsden equation without boundary in certain contact metric manifolds. Therefore, Theorem 3.3 \cite{3b} and Theorem 3.2 \cite{14}  are particular cases of Corollary \ref{c2}.
\end{rem}
\begin{thm}\label{t2}
Let $M^{2n+1}(\varphi,\eta,\xi,g)$ be a non-Sasakian $(\kappa,\mu)$-contact metric manifold and $f$ a non-trivial smooth function satisfying (\ref{A1}), then $M$ is locally flat for $n=1$ and for $n>1$ is locally isometric to the product space $E^{n+1}\times S^n(4).$
\end{thm}
\begin{proof}
Replacing $Y$ by $\xi$ in (\ref{a8}), then using it in (\ref{m1}) gives
\begin{eqnarray}\label{3.1}
\nabla_\xi h=\mu h\varphi.
\end{eqnarray}
Taking the covariant derivative of (\ref{a9}) along $\xi$ and using (\ref{3.1}) we get
\begin{eqnarray}\label{3.2}
(\nabla_\xi Q)X=\mu(2(n-1)+\mu)h\varphi X,
\end{eqnarray}
for any vector field $X$ on $M$. Moreover, from (\ref{a8}) we have $Q\xi=2n\kappa\xi$. Differentiating this along arbitrary vector field $X$ and using (\ref{a1}) yields
\begin{eqnarray}\label{3.3}
(\nabla_XQ)\xi=Q(\varphi+\varphi h)X-2n\kappa(\varphi+\varphi h)X.
\end{eqnarray}
Taking an inner product of Lemma \ref{l1} with $\xi$ and using (\ref{3.2}) and (\ref{3.3}), we obtain
\begin{eqnarray}\label{3.4}
g(R(X,Y)Df,\xi)=2n\kappa\{(Xf)\eta(Y)-(Yf)\eta(X)\}\nonumber\\+f\{g(Q\varphi X+\varphi QX,Y)+g(Q\varphi hX+h\varphi QX,Y)\}\nonumber\\-4n\kappa fg(\varphi X,Y)+(X\psi)\eta(Y)-(Y\psi)\eta(X).
\end{eqnarray}
Contracting Lemma \ref{l1} along $X$ and making use of $Q\xi=2n\kappa\xi$ yields
\begin{eqnarray}\label{3.5}
RDf+2nD\psi=0.
\end{eqnarray}
On the other hand, replacing $X$ by $\xi$ in (\ref{3.4}) gives
\begin{eqnarray}\label{3.6}
g(R(\xi,Y)Df,\xi)=2n\kappa\{(\xi f)\eta(Y)-(Yf)\}+(\xi\psi)\eta(Y)-(Y\psi).
\end{eqnarray}
Substituting $X=\xi$ in (\ref{a8}), then taking its inner product with $Df$, we get
\begin{eqnarray}\label{3.7}
g(R(\xi,Y)\xi,Df)=\kappa\{(\xi f)\eta(Y)-(Yf)\}-\mu g(hY,Df).\nonumber
\end{eqnarray}
In consequence of this, Eq. (\ref{3.6}) becomes
\begin{eqnarray}\label{3.8}
(2n+1)\kappa\{(\xi f)\xi-Df\}+\{(\xi\psi)\xi-D\psi\}-\mu hDf=0.
\end{eqnarray}
Replacing $X$ by $\xi$ in (\ref{2.1}) and using $Q\xi=2n\kappa\xi$ infer
\begin{eqnarray}\label{3.9}
\nabla_\xi Df=(2n\kappa f+\psi)\xi.
\end{eqnarray}
Combining (\ref{3.5}) and (\ref{3.8}), then taking its covariant derivative along $\xi$ and using (\ref{3.1}) and (\ref{3.9}), we get
\begin{eqnarray}\label{3.10}
[R-2n(2n+1)\kappa][2n\kappa f\xi +\psi\xi-\xi(\xi f)\xi]-2n\mu^2h\varphi Df=0.
\end{eqnarray}
Operating (\ref{3.10}) by $\varphi$ gives $2n\mu^2hDf=0$. Then, operating this by $h$ and using (\ref{a5}), we get $2n\mu^2(\kappa-1)\varphi^2 Df=0$. Since $M$ is non-Sasakian, we have either (i) $\mu=0$ or (ii) $\varphi^2 Df=0$.\\

\noindent Case (i): Replacing $X$ by $\varphi X$ and $Y$ by $\varphi Y$ in (\ref{3.4}) and noting that $f\neq0$, we get
\begin{eqnarray}\label{3.11}
Q\varphi X+\varphi QX-\varphi QhX-hQ\varphi X-4n\kappa\varphi X=0.
\end{eqnarray}
Replacing $X$ by $\varphi X$ in (\ref{a9}), we get
$$Q\varphi X=[2(n-1)-n\mu]\varphi X+[2(n-1)+\mu]h\varphi X.$$ Then, operating this by $h$ and using (\ref{a5}) infer
$$hQ\varphi X=[2(n-1)-n\mu]h\varphi X-(\kappa-1)[2(n-1)+\mu]\varphi X.$$ Again operating (\ref{a9}) by $\varphi$ yields
$$\varphi QX=[2(n-1)-n\mu]\varphi X+[2(n-1)+\mu]\varphi hX.$$ Replacing $X$ by $hX$ in the last expression gives
$$\varphi QhX=[2(n-1)-n\mu]\varphi hX-(\kappa-1)(n-1)+\mu]\varphi X.$$
Making use of the last four equations in (\ref{3.11}), we obtain 
\begin{eqnarray}\label{m2}
\mu(n+1)=\kappa(\mu-2).
\end{eqnarray}
By our assumption, $\mu=0$ implies from above relation $\kappa=0$. Therefore, (\ref{a8}) becomes $R(X,Y)\xi=0$. Hence we can conclude that $M$ is locally flat for $n=1$ and for $n>1$ is locally isometric to the product space of $E^{n+1}\times S^n(4)$ (see \cite{7a}).\\

\noindent Case (ii): $\varphi^2Df=0$ implies $Df=(\xi f)\xi$. Taking the covariant derivative of this along vector field $X$ and using (\ref{2.1}) and (\ref{a3}), we get
\begin{eqnarray}\label{3.12}
\nabla_XDf=X(\xi f)\xi-(\xi f)(\varphi X+\varphi hX).\nonumber
\end{eqnarray}
Anti-symmetrizing the last expression yields $$X(\xi f)\eta(Y)-Y(\xi f)\eta(X)+(\xi f)d\eta(X,Y)=0.$$ Choosing $X,Y\perp\xi$ and noting that $d\eta\neq0$, we obtain $\xi f=0$ implies $Df=0$, i.e., $f$ is constant on $M$. Also, from (\ref{3.5}) we see that $\psi$ is constant. In consequence, (\ref{2.1}) implies $M$ is Einstein, i.e., $QX=2n\kappa X$. Contracting this we get the scalar curvature as $R=2n\kappa(2n+1)$. Comparing this with (\ref{a11}) it follows that $n\mu=2(n-1)-2n\kappa$. Making use of the last relation and $QX=2n\kappa X$ in (\ref{a9}), we obtain $(2(n-1)+\mu)h=0$. Since $M$ is non-Sasakian we must have $2(n-1)+\mu=0$. Obviously for $n=1$, $\mu=0=\kappa$ and hence $R(X,Y)\xi=0$. On the other hand, for $n>1$, inserting $2(n-1)+\mu=0$ in (\ref{m2}) we obtain $\kappa=n-\frac{1}{n}>1$, a contradiction. This completes the proof.
\end{proof}
\noindent By similar arguments as in Corollary \ref{c2}, we can state the following:
\begin{cor}\label{c3}
Let $M^{2n+1}(\varphi,\eta,\xi,g)$ be a non-Sasakian $(\kappa,\mu)$-contact metric manifold without boundary and $f$ is a non-constant solution of $V$-static equation (\ref{1.4}), then it is locally flat for $n=1$ and for $n>1$ is locally isometric to the product space $E^{n+1}\times S^n(4)$.
\end{cor}
Next, we construct an example of $(\kappa,\mu)$-contact metric manifold and $K$-contact manifold in which a non-trivial smooth function $f$ satisfies (\ref{A1}).
\begin{ex}
In \cite{G1}, the authors constructed a 3-dimensional $(1-\lambda^2,0)$-contact metric manifold $M=\{(x,y,z)\in\mathbb{R}^3:x\neq0\}$ in the Euclidean space $\mathbb{R}^3$, where $\lambda$ is a real number. Take $f(x,y,z)=x^2$, then one can easily see that the equations (\ref{2.1}) holds for $\psi=2(1-x^2)$ on $M$. In particular, taking $\lambda=0$ in the above example give a 3-dimensional $K$-contact manifold. Moreover, for $f(x,y,z)=x^2$ and  $\psi=2(1-x^2)$, the 3-dimensional $K$-contact manifold satisfies (\ref{A1}). The expression for Ricci operator is obtained as $QX=2X$, that is, $M$ is Einstein (see \cite{G1}) with constant scalar curvature $R=6$. Thus Theorem (\ref{t1}) is verified.
\end{ex}

%\Acknowledgements{}{\noindent We would like to thank ...}
%\Funding{}{\noindent This work is supported by ...}

%% ACKNOWLEDGEMENTS

%\normalsize\section*{\large Acknowledgements}\vspace{-0.3cm}
%\noindent The authors would like to express their sincere thanks to the editor and the anonymous reviewers for their helpful comments and suggestions.\vspace{-0.6cm}
% \section*{\large Funding}\vspace{-0.3cm}
%\noindent There is no funding for this work.
%  \section*{\large Availability of data and materials}\vspace{-0.3cm}
%\noindent Not applicable.\vspace{-0.6cm}
% \section*{\large Competing interests}\vspace{-0.3cm}
%\noindent The authors declare that they have no competing interests.\vspace{-0.6cm}
%\section*{\large Author’s contributions}\vspace{-0.3cm}
%\noindent All authors contributed equally to the writing of this paper. All authors read and approved the final manuscript. \vspace{-0.6cm}

%%	BIBLIOGRAPHY
%% 

%% General Considerations Regarding References

%% The author is responsible for the accuracy of the references. All listed references must be cited in the text.

%% Articles which have been accepted for publication should be listed with the expected year of publication (if known) and name of the journal. In place of volume and page numbers, add (in press) at the end.

%% Reference Style and Format

%% The list of references at the end is arranged in numerical order. The names of all authors should be listed. References by the same author or group of authors should be listed in chronological order.

%% References should be styled as follows, with correct punctuation:

\end{document}